\documentclass[a4paper,11pt]{article}
\usepackage{makeidx}
\usepackage[mathscr]{eucal}
\usepackage[dvipdf]{color} 
\usepackage{amssymb, amsfonts, amsmath, amsthm, graphics}

\newtheorem{proposition}{Proposition}[section]
\newtheorem{theorem}{Theorem}[section]
\newtheorem{lemma}[proposition]{Lemma}
\newtheorem{remark}{Remark}[section]
\newtheorem{corollary}[theorem]{Corollary}

\numberwithin{equation}{section}

\title{
Linear instability and nondegeneracy of ground state for combined power-type nonlinear scalar field equations with the Sobolev critical exponent and large frequency parameter
}
\author{Takafumi Akahori, Slim Ibrahim, Hiroaki Kikuchi}
\date{\today}
\newcommand{\TPho}{\widetilde{\Phi}_\omega}

\newcommand{\dH}{\dot{H}^1}

\begin{document}
\maketitle

\begin{abstract}
We consider combined power-type nonlinear scalar field equations with the Sobolev critical exponent. In \cite{AIKN3}, it was shown that if the frequency parameter is sufficiently small, then the positive ground state is nondegenerate and linearly unstable, together with an application to a study of global dynamics for nonlinear Schr\"odinger equations. In this paper, we prove the nondegeneracy and   linear instability of the ground state frequency for sufficiently large frequency parameters. Moreover, we show that the derivative of the mass of ground state with respect to the frequency is negative.    
\end{abstract}

%%%%%%%%%%%%%%%%%%%%%%%%%%%%%%%%%%%%%%%%%%%%%%%%%%%%%%%

\section{Introduction}

In this paper, we study the linear instability and the nondegeneracy of ground state to the scalar field equation of the form 
\begin{equation}\label{eq:1.1} 
\omega u  - \Delta u 
-|u|^{p-1}u-|u|^{\frac{4}{d-2}}u =0 
\quad 
\mbox{in $\mathbb{R}^{d}$},
\end{equation}
where $d\ge 3$, $\omega >0$ and $1< p < \frac{d+2}{d-2}$. These subjects (linear instability/nondegeneracy) are related to the study of the ``global dynamics around ground state'' for the corresponding evolution equations (see \cite{AIKN3, Nakanishi-Schlag1, Nakanishi-Schlag2}); evolution equations corresponding to \eqref{eq:1.1} are the nonlinear Schr\"{o}dinger equation   
\begin{equation} \label{nls}
i \frac{\partial \psi}{\partial t} + \Delta \psi + |\psi|^{p-1} \psi 
+ |\psi|^{\frac{4}{d-2}} \psi = 0 
\qquad \mbox{in $\mathbb{R}^{d} \times \mathbb{R}$}, 
\end{equation} 
and the nonlinear Klein-Gordon equation 
\begin{equation} \label{kg}
\frac{\partial^{2} \psi}{\partial t^{2}} - \Delta \psi + \psi -|\psi|^{p-1} \psi -|\psi|^{\frac{4}{d-2}} \psi = 0 
\qquad \mbox{in $\mathbb{R}^{d} \times \mathbb{R}$}.
\end{equation}
Furthermore, for various equations, the linear instability and nondegeneracy of ground state has been studied in connection with standing waves 
 (see, e.g., \cite{Georgiev-Ohta, Grillakis, Grillakis2, Grillakis-Shatah-Strauss2, Jones, Mizumachi1, Mizumachi2, Shatah-Strauss2} for the linear instability
and \cite{Grossi, Kabeya-Tanaka, Killip-Oh-Pocivnicu-Visan, 
Maris, Weinstein} for the nondegeneracy). 
Here, by a standing wave, we mean a solution of the form $\psi(t, x) = e^{i \omega t} u(x)$ for some  $\omega >0$ and some function $u$ on $\mathbb{R}^{d}$.

We shall make clear what the ground state means. To this end, we introduce a functional $\mathcal{S}_{\omega}$ as 
\begin{equation}\label{eq:1.5}
\mathcal{S}_{\omega}(u)
:=
\frac{\omega}{2}\|u\|_{L^{2}}^{2}
+
\frac{1}{2}\|\nabla u \|_{L^{2}}^{2}
-
\frac{1}{p+1}\|u\|_{L^{p+1}}^{p+1}
-
\frac{1}{2^{*}}\|u\|_{L^{2^{*}}}^{2^{*}}
,
\end{equation}
where 
\begin{equation}\label{eq:1.6}
2^{*}:=\frac{2d}{d-2}.  
\end{equation}
This functional $\mathcal{S}_{\omega}$ is called the action associated with \eqref{eq:1.1}. Then, by a ground state, we mean the least action solution among all nontrivial solutions  to \eqref{eq:1.1} in $H^{1}(\mathbb{R}^{d})$. 
\par 
For most of the scalar field equations, a ground state can be obtained as a minimizer of some variational (minimization) problem of the associate action. Furthermore, the variational value introduces an invariant set. 
The study of solutions in 
 such an invariant set has been studied by many researchers 
(see \cite{Akahori-Ibrahim-Kikuchi-Nawa1, AIKN2, AIKN3, CG, D-H-R, DM, Kenig-Merle, MXZ, Nakanishi-Roy, Nakanishi-Schlag1, Nakanishi-Schlag2}).

Now, we refer to the the existence of ground state to our equation \eqref{eq:1.1}. The following result is known (see, e.g., Proposition 1.1 in \cite{AIIKN}):  \begin{proposition}\label{proposition:1.1}
Assume either $d= 3$ and $3<p<5$, or else $d\ge4$ and $1<p<\frac{d+2}{d-2}$. Then, for any $\omega >0$ there exists a ground state to \eqref{eq:1.1}. 
\end{proposition}

\begin{remark}\label{18/09/05/15:30}
When $d=3$ and $1<p\le 3$, we can prove the existence of ground state for small frequencies (see, e.g., \cite{AIKN3}): More precisely, there exists $\omega_{0}>0$ such that for any $0< \omega < \omega_{0}$, the equation \eqref{eq:1.1} admits a ground state. We do not know the existence of ground state to \eqref{eq:1.1} for a large $\omega$ in three dimensions. 
\end{remark}

The standard theory for semilinear elliptic equation gives us the following information on the ground state (see \cite{AIKN3, GNN, Lieb-Loss}): 
\begin{lemma}\label{17/07/16/16:37}
Assume $d\ge 3$ and $1<p<\frac{d+2}{d-2}$. Then, for any $\omega >0$, the following holds as long as a ground state exists: any ground state $Q_{\omega}$ to \eqref{eq:1.1} is of class $C^{2}$ on $\mathbb{R}^{d}$, and there exist $y\in \mathbb{R}^{d}$, $\theta \in \mathbb{R}$ and a positive ground state $\Phi_{\omega}$ to \eqref{eq:1.1} such that $Q_{\omega}(x)=e^{i\theta} \Phi_{\omega}(x-y)$ for all $x\in \mathbb{R}^{d}$. Furthermore, any positive ground state $\Phi_{\omega}$ is strictly decreasing in the radial direction, and there exist $C(\omega)>0$ and $\delta(\omega)>0$ such  that 
\begin{equation}\label{17/08/13/11:05}
|\Phi_{\omega}(x)|+|\nabla \Phi_{\omega}(x)| 
\le C(\omega) e^{-\delta(\omega)|x|}
.
\end{equation}
\end{lemma}

We refer to the uniqueness of ground state to \eqref{eq:1.1} in a remark:
\begin{remark}\label{18/06/24/15:47}
\noindent 
{\rm (i)}~Assume $3 \leq d \leq 6$ and 
$\frac{4}{d-2} \leq p < \frac{d+2}{d-2}$. 
 Then, we can verify that the result of Pucci and Serrin~\cite{PS} is applicable (see \cite[Appendix C]{AIIKN}), and find that a positive solution to \eqref{eq:1.1} is unique. 
\\
\noindent 
{\rm (ii)}~Assume $d \ge 5$ and $1<p< \frac{d+2}{d-2}$. Then in \cite{AIIKN}, it was proved that there exists $\omega_{1}>0$ such that for any $\omega>\omega_{1}$, the positive radial ground state to \eqref{eq:1.1} is unique.  
\\
\noindent 
{\rm (iii)}~Assume $d=3$ and $3<p<5$. Then, it follows from the result by Coles and Gustafson \cite{CG} that for any sufficiently large $\omega>0$, the positive ground state to \eqref{eq:1.1} is unique. 
\end{remark}

We find from Lemma 19 and Lemma 20 of \cite{Shatah-Strauss} that 
 the mapping $\omega \in (\omega_{3},\infty) \mapsto \Phi_{\omega} \in H^{1}(\mathbb{R}^{d})$ is continuously differentiable, where $\Phi_{\omega}$ is the unique positive ground sate to \eqref{eq:1.1}.
\par 
In order to state our results, we need to introduce several notation: 
\\ 
\noindent 
\textbullet~We use $\omega_{1}$ to denote the frequency such that for any $\omega> \omega_{1}$, a positive ground state to \eqref{eq:1.1} is unique .  
\\
\noindent
\textbullet~We use $\Phi_{\omega}$ to denote a positive ground state to \eqref{eq:1.1}. 
\\ 
\noindent 
\textbullet~The symbol $\mathcal{L}_{\omega}$ denotes the linearized operator around $\Phi_{\omega}$ from $H^{2}(\mathbb{R}^{d})$ to $L^{2}(\mathbb{R}^{d})$, namely for any $u\in H^{2}(\mathbb{R}^{d})$, 
\begin{equation}\label{15/02/02/10:25}
\mathscr{L}_{\omega}u
=
\omega u -\Delta u 
-
\frac{p+1}{2}\Phi_{\omega}^{p-1}u 
-\frac{p-1}{2}\Phi_{\omega}^{p-1}\overline{u}
-\frac{d}{d-2}\Phi_{\omega}^{\frac{4}{d-2}}u 
-\frac{4}{d-2}\Phi_{\omega}^{\frac{4}{d-2}}\overline{u}
.
\end{equation}

\noindent 
\textbullet~We introduce two operators $L_{\omega,+}$ and $L_{\omega,-}$ as 
\begin{align}\label{13/02/02/17:56}
L_{\omega, +} 
&:= 
\omega  -\Delta  
-p \Phi_{\omega}^{p-1} -\frac{d+2}{d-2}\Phi_{\omega}^{\frac{4}{d-2}}, 
\\[6pt]
\label{13/02/02/17:57}
L_{\omega, -} 
&:= 
\omega  -\Delta  - \Phi_{\omega}^{p-1} - \Phi_{\omega}^{\frac{4}{d-2}}.
\end{align}

%\vspace{1zw}

In the study of global dynamics around the ground state for \eqref{nls}, 
fine properties of the ground state are needed. 
On the other hand, the existence of the Sobolev critical exponent and the failure of scale invariance make the study of ground state complicated. A significance of this paper is that we deal with the equation \eqref{eq:1.1} having both such difficulties.

First, we refer to the linearized operator. It is worthwhile noting that the operator $\mathcal{L}_{\omega}$ multiplied by the imaginary unit $i$, namely  $i \mathcal{L}_{\omega}$, plays an important role in the study of dynamics around ground state for the nonlinear Schr\"{o}dinger equation \eqref{nls} (see \cite{AIKN3, DM, Nakanishi-Roy, Nakanishi-Schlag1}). We say that the standing wave 
$e^{i \omega t}$ for \eqref{nls} is linearly unstable if the linearized operator $-i\mathcal{L}_{\omega}$ has an eigenvalue with positive real part. One of our main results is related to the linear instability of ground state: 
\begin{theorem}\label{17/08/03/11:47}
Assume $d= 3$ and $3<p <5$; or $d=4$ and $2\le p <3$; or $d\ge 5$ and $1<p<\frac{d+2}{d-2}$. Then, there exists $\omega_{2}>\omega_{1}$ such that for any $\omega >\omega_{2}$, the operator $-i\mathcal{L}_{\omega}$ has a positive eigenvalue as an operator in $L_{real}^{2}(\mathbb{R}^{d})$. 
\end{theorem} 

We give a proof of Theorem \ref{17/08/03/11:47} in Section \ref{17/08/03/10:55}.

The second main result relates to the nondegeneracy of ground state. Here, let us recall that $\Phi_{\omega}$ is said to be nondegenerate in $H^{1}_{\rm rad}(\mathbb{R}^{d})$ if the linear problem $L_{\omega,+}u=0$ has no nontrivial solution  in $H^{1}_{\rm rad}(\mathbb{R}^{d})$, namely, ${\rm Ker} L_{\omega,+} 
 \big|_{H^{1}_{\rm rad}(\mathbb{R}^{d})} = \{0\}$. 

\begin{theorem}\label{18/09/09/17:09}
Assume that $d =3$ and $3 < p  <5$. Then, there exists $\omega_{3}>0$ 
 such that for any $\omega > \omega_{3}$, the positive ground state to \eqref{eq:1.1} is nondegenerate in $H_{\rm rad}^{1}(\mathbb{R}^{d})$. 
\end{theorem}

\begin{remark}\label{18/09/10/19:31}
In Theorem 1.1 of \cite{AIIKN}, it was proved that 
  if $d\ge 5$ and $1<p<\frac{d+2}{d-2}$, then there exists $\omega_{3}>0$ such that for any $\omega>\omega_{3}$, the positive ground state $\Phi_{\omega}$ to \eqref{eq:1.1} is nondegenerate. When $d=4$, the nondegeneracy of  ground state to \eqref{eq:1.1} is still open.
\end{remark}

The proof is basically the same as the one of Theorem 1.1 in \cite{AIIKN}. The main difference between our Theorem \ref{18/09/09/17:09} and Theorem 1.1 
in \cite{AIIKN} is the integrability of the Talenti function (see \eqref{eq:1.16} below). More precisely, in \cite{AIIKN}, the case $d\ge 5$ is only dealt with, and the fact that $W \in L^{2}(\mathbb{R}^{d})$ was essential. On the other hand, in our case $d=3$, the Talenti function is not in $L^{2}(\mathbb{R}^{3})$, which gives rise to some difficulty in an application of the argument of \cite{AIIKN}. To overcome this difficulty, we employ the resolvent expansion (see Lemma \ref{thm-resolvent-3} below). We give a proof of Theorem \ref{18/09/09/17:09} in Section \ref{18/09/09/17:30}.\par 
The last main result of this paper relates to the sufficient condition of instability for standing waves obtained by Grillakis, Shatah and Strauss~\cite{Grillakis-Shatah-Strauss1}. More precisely, they gave a sufficient condition of the orbital stability for standing waves to a general Hamiltonian system under 
some assumption of the linearized operator. For nonlinear Schr\"odinger equations, the condition is described by the derivative of the mass of standing wave with respect to the frequency parameter; the positivity and the negativity imply the stability and the instability, respectively. Moreover, the derivative of the mass also plays an important role in ``symplectic decomposition'' which is used in the analysis of the solutions near the standing wave (see, e.g., (2.17) and (2.18) in \cite{Nakanishi-Schlag2}).

We state the last main result of this paper: 
\begin{theorem}\label{17/08/15/11:05}
Assume $d= 3$ and $3<p <5$; or $d\ge 5$ and $1<p<\frac{d+2}{d-2}$. Then, there exists $\omega_{4}>\omega_{3}$ 
($\omega_{3}$ is the constant give in Theorem \ref{18/09/09/17:09})
 such that for any $\omega \in (\omega_{3},\infty)$, 
\begin{equation}\label{13/03/30/10:40}
\frac{d}{d\omega}\| \Phi_{\omega}\|_{L^{2}}^{2} <0
.
\end{equation}
\end{theorem}

Note here that for the single power equation 
\begin{equation}\label{18/09/17/10:55}
- \Delta u + \omega u - |u|^{p-1} u = 0 ,  
\end{equation} 
the ground state is obtained as the scaling of the one for $\omega=1$: This is due to the scale invariance of the equation.  Thus, it is easy to compute the derivative of the mass with respect to $\omega$. On the other hand, when the scale invariance of an equation is lost, it would be much harder to find the derivative of the mass even for the combined power-type equation \eqref{eq:1.1}. An idea used in \cite{AIKN3} is a perturbation argument comparing \eqref{eq:1.1} with \eqref{18/09/17/10:55}. It works well when $\omega$ is sufficiently small. However,  the limiting equation as $\omega \to \infty$ is \eqref{eq:1.16} below, and the ground state to \eqref{eq:1.16} (known as the Talenti function) is no longer in $L^{2}(\mathbb{R}^{d})$ for $d=3,4$. Thus, we cannot apply the same argument 
 as \cite{AIKN3} to the case where $\omega$ is large. To overcome this difficulty, we give an argument combining the ones in \cite{Grillakis, Grillakis-Shatah-Strauss1} and the nondegeneracy of ground state (Theorem \ref{18/09/09/17:09} an Remark \ref{18/09/10/19:31}): Due to the nondegeneracy, 
 we cannot deal with the four dimensions.  We give a proof of Theorem \ref{17/08/15/11:05} in Section \ref{17/08/15/11:02}. 
\par 
Finally, we remark that for any sufficiently small frequencies, the ``9-set theory'' for \eqref{nls} is proved in \cite{AIKN3}. Combining the arguments in \cite{AIKN3} and Theorem \ref{17/08/03/11:47} through Theorem \ref{17/08/15/11:05} in this paper, we could obtain the 9-set theory for a sufficiently large frequencies: The result will be given elsewhere.  

%\vspace{1zw}
In addition to the above notation, we introduce another one:
\\
\noindent 
\textbullet~$L_{real}^{2}(\mathbb{R}^{d})$ denotes the set of functions in $L^{2}(\mathbb{R}^{d})$ equipped with the following inner product:   
\begin{equation}\label{17/08/15/06:01}
\langle u,v \rangle
:=
\Re \int_{\mathbb{R}^{d}}u(x)\overline{v(x)} 
\,dx \quad \mbox{$u,v \in L^{2}(\mathbb{R}^{d})$}.
\end{equation}
\noindent 
\textbullet~We use the convention that if $u \in H^{-1}(\mathbb{R}^{d})$ and $v\in H^{1}(\mathbb{R}^{d})$, then $\langle u, v \rangle$ denotes the duality pairing of $H^{1}(\mathbb{R}^{d})$ and $H^{-1}(\mathbb{R}^{d})$, namely 
\begin{equation}\label{17/08/15/06:03}
\langle u, v \rangle
=
\Re \int_{\mathbb{R}^{d}} 
(1-\Delta)^{-\frac{1}{2}}u(x) 
\overline{(1-\Delta)^{\frac{1}{2}} v(x)}\,dx .
\end{equation}

\noindent 
\textbullet~We define 
\begin{align}
\label{eq:1.11}
M_{\omega}&:=\Phi_{\omega}(0)
,
\\[6pt]
\label{eq:1.10}
\widetilde{\Phi}_{\omega}(x)&:=M_{\omega}^{-1}\Phi_{\omega}(M_{\omega}^{-\frac{2}{d-2}}x)
\end{align}

\noindent 
\textbullet~We use $W$ to denote the Talenti function with $W(0)=1$, namely  
\begin{equation}\label{eq:1.16}
W(x)
:=\Big(1+\frac{|x|^{2}}{d(d-2)} \Big)^{-\frac{d-2}{2}}
\end{equation}

\noindent 
\textbullet~The linearized operator around $W$ is denoted by $L_{+}$:
\begin{equation}
\label{eq:1.18}
L_{+}:=-\Delta -\frac{d+2}{d-2}W^{\frac{4}{d-2}}
\end{equation}

\noindent 
\textbullet~
\begin{equation}\label{eq:1.19}
\Lambda W:= \frac{d-2}{2}W+x\cdot \nabla W
\end{equation}

\noindent 
\textbullet~We use $B_{R}$ to denote the closed ball in $\mathbb{R}^{d}$ of the center $0$ and a radius $R$:
\begin{equation}\label{17/08/15/11:29}
B_{R}:=\{ x\in \mathbb{R}^{d} \colon |x|\le R \}.
\end{equation}

\noindent 
\textbullet~For given positive quantities $a$ and $b$, the notation $a \lesssim b$ means the inequality $a\le C b$ for some positive constant; the constant $C$  depends only on $d$ and $p$, unless otherwise noted.

%%%%%%%%%%%%%%%%%%%%%%%%%%%%%%%%%%%%%%%%%%%%%%%%%%%%%%%%%%%%%%%%%%%%%%%%%%%%%%%%
\section{Basic properties of ground state}\label{section:2}

In this section, we give basic properties of the ground state $\Phi_{\omega}$. 
\par 
When $\omega$ is large, it would be natural to compare the equation \eqref{eq:1.1} with  
\begin{equation}\label{eq:1.15}
\Delta u +|u|^{\frac{4}{d-2}}u=0. 
\end{equation}
Note here that the Talenti function $W$ is a solution to \eqref{eq:1.15} in $\dot{H}^{1}(\mathbb{R}^{d})$. Moreover, it is known (see, e.g, \cite{AIIKN}) that we need to consider the rescaling $\widetilde{\Phi}_{\omega}$ (see \eqref{eq:1.10}), in stead of $\Phi_{\omega}$, for such a comparison. In particular, we easily see from Lemma \ref{17/07/16/16:37} that  
\begin{equation}\label{eq:2.13}
\|\widetilde{\Phi}_{\omega}\|_{L^{\infty}}
=
\widetilde{\Phi}_{\omega}(0)=1=\|W\|_{L^{\infty}}=W(0).
\end{equation}
We also verify that $\widetilde{\Phi}_{\omega}$ satisfies 
\begin{equation}\label{eq:2.1}
-\Delta \widetilde{\Phi} 
+ 
\alpha_\omega \widetilde{\Phi} 
- 
\beta_{\omega} \widetilde{\Phi}^{p} - \widetilde{\Phi}^{\frac{d+2}{d-2}} = 0, 
\end{equation}
where 
\begin{equation}\label{eq:1.14}
\alpha_{\omega}:=\omega M_{\omega}^{-\frac{4}{d-2}},
\qquad 
\beta_{\omega}:= M_{\omega}^{p-1-\frac{4}{d-2}}
.
\end{equation}
The following result given in Lemma 2.3 of \cite{AIIKN} tells us the asymptotic behavior of $M_{\omega}$, $\alpha_\omega$ and $\beta_{\omega}$ as $\omega \to \infty$: 
\begin{lemma}\label{proposition:2.3}
Assume $d= 3$ and $3<p <5$; or $d=4$ and $2\le p<3$; or $d\ge 5$ and $1<p<\frac{d+2}{d-2}$. Then 
\begin{align}
\label{eq:2.16}
&\lim_{\omega \to \infty} M_{\omega}= \infty,
\\[6pt]
\label{eq:2.17}
&\lim_{\omega \to \infty} \alpha_{\omega} 
= 
\lim_{\omega \to \infty} \beta_{\omega}
=0.
\end{align} 
\end{lemma}

We expect from \eqref{eq:2.13}, \eqref{eq:2.1} and Lemma \ref{proposition:2.3} that $\widetilde{\Phi}_{\omega}$ converges to $W$ as $\omega\to \infty$. This is true and given in Proposition 2.1 of \cite{AIIKN}:  
\begin{lemma}\label{18/09/05/01:01}
Assume $d= 3$ and $3<p <5$; or $d=4$ and $2\le p<3$; or $d\ge 5$ and $1<p<\frac{d+2}{d-2}$. Then, it holds that 
\begin{equation}\label{eq:2.2}
\lim_{\omega \to \infty} 
\big\| \TPho - W  \big\|_{\dH} = 0.
\end{equation}
\end{lemma}

We also have the following uniform decay estimate (Proposition 3.1 of \cite{AIIKN}): 
\begin{lemma}\label{18/09/05/01:05}
Assume $d=3$ and $3<p <5$, or $d\ge 4$ and $1<p<\frac{d+2}{d-2}$. Then, there exist $\omega_{dec}>0$ and  $C_{dec}>0$ such that for any $\omega >\omega_{dec}$ and any $x\in \mathbb{R}^{d}$, 
\begin{equation}\label{eq:3.1}
\widetilde{\Phi}_{\omega}(x) \le C_{dec} \left(1 +|x| \right)^{-(d-2)}
. 
\end{equation}
\end{lemma}  
We can derive the  following convergence result from Lemma \ref{18/09/05/01:01} and Lemma \ref{18/09/05/01:05} (see Corollary 3.1 of \cite{AIIKN}):
\begin{corollary}\label{theorem:3.1}
Assume $d= 3$ and $3<p <5$; or $d=4$ and $2\le p<3$; or $d\ge 5$ and $1<p<\frac{d+2}{d-2}$. Then, for any $q> \frac{d}{d-2}$, it holds that  
\begin{equation}\label{eq:3.2}
\lim_{\omega \to \infty} \| \widetilde{\Phi}_{\omega} -W \|_{L^{q}}=0
.
\end{equation}
\end{corollary}

%%%%%%%%%%%%%%%%%%%%%%%%%%%%%%%%%%%%%%%%%%%%%%%%%%%%%%%%%%%%%%%%%%%%%%%%%%%%%%%%
\section{Proof of Theorem \ref{17/08/03/11:47}}
\label{17/08/03/10:55}

In this section, we give a proof of Theorem \ref{17/08/03/11:47}, namely we show the existence of positive eigenvalue of $-i\mathscr{L}$. 
\par 
We introduce notation used in this section: 
\\
\noindent 
{\bf Notation}~We use $C(\omega)$ to denote several large positive constant depending only on $d$, $p$ and $\omega$ which may vary from line to line. Moreover,  $c(\omega)$ means $1/C(\omega)$. 

%\vspace{1zw}

We remark that the operator $-i \mathscr{L}_{\omega}$ is not self-adjoint in $L_{real}^{2}(\mathbb{R}^{d})$. On the other hand, the operators $L_{\omega,+}$ and $L_{\omega,-}$ (see \eqref{13/02/02/17:56} and \eqref{13/02/02/17:57}) are self-adjoint in $L_{real}^{2}(\mathbb{R}^{d})$. We have the relationship  
\begin{equation}\label{13/02/20/9:45}
\mathscr{L}_{\omega}u
=
L_{\omega, +}\Re[u]+ i L_{\omega,-}\Im[u] .
\end{equation}
In order to prove Theorem \ref{17/08/03/11:47}, we need the following lemma:\begin{lemma}\label{13/03/18/15:50}
Assume $d= 3$ and $3<p <5$; or $d\ge 4$ and $1<p<\frac{d+2}{d-2}$. Then, for any $\omega >0$, the operator $L_{\omega,-}$ is  non-negative,  and 
\begin{equation}\label{18/09/06/02:52}
{\rm Ker}\, L_{\omega,-}={\rm span}\{ \Phi_{\omega}\}.
\end{equation} 
Furthermore, there exists $c(\omega)>0$ depending only on $d$, $p$ and $\omega$ such that for any nontrivial function $u\in H^{1}(\mathbb{R}^{d})$ with $(u,\Phi_{\omega})_{L_{real}^{2}}=0$, 
\begin{equation}\label{13/04/12/12:19}
\langle L_{\omega,-}u,u \rangle
\ge 
c(\omega) \| u \|_{H^{1}}^{2}.
\end{equation}
\end{lemma}
Lemma \ref{13/03/18/15:50} can be proved by using Lemma 8.1 of \cite{Cazenave}, the positivity of $\Phi_{\omega}$, $L_{\omega,-}\Phi_{\omega}=0$, $\lim_{|x|\to \infty}\Phi_{\omega}(x)=0$ and Weyl's essential spectrum theorem.

We see from Lemma \ref{13/03/18/15:50} that $L_{\omega,-}$ has a unique square roof $L_{\omega,-}^{\frac{1}{2}}$ with domain $H^{1}(\mathbb{R}^{d})$. 
\par 
We introduce operators related to the rescaled equation \eqref{eq:1.1}: 
\begin{align}
\label{13/03/30/11:40}
&\widetilde{L}_{\omega,+}
:= 
-\Delta 
+
\alpha_{\omega}
-
p \beta_{\omega} \widetilde{\Phi}_{\omega}^{p-1}
-
\frac{d+2}{d-2} \widetilde{\Phi}_{\omega} 
^{\frac{4}{d-2}},
\\[6pt]
\label{13/04/12/14:40}
&\widetilde{L}_{\omega,-}
:= 
-\Delta +\alpha_{\omega} - \beta_{\omega} \widetilde{\Phi}_{\omega}^{p-1}
- \widetilde{\Phi}_{\omega}^{\frac{4}{d-2}}.
\end{align}
Furthermore, we define $\widetilde{\mathscr{L}}_{\omega}$ to be that for any $u \in H^{2}(\mathbb{R}^{d})$, 
\begin{equation}\label{18/09/06/20:33}
\widetilde{\mathscr{L}}_{\omega}u
=
\widetilde{L}_{\omega,+}\Re[u] + i \widetilde{L}_{\omega,-}\Im[u].
\end{equation}
Note that if we have the relationship $f(x)= M_{\omega} \widetilde{f} (M_{\omega}^{\frac{2}{d-2}}x)$ between functions $f$ and $\widetilde{f}$, then  
\begin{equation}\label{18/09/06/20:42}
\mathscr{L}_{\omega}f(x)
=
M_{\omega}^{\frac{d+2}{d-2}}  
\widetilde{\mathscr{L}}_{\omega} \widetilde{f} (M_{\omega}^{\frac{2}{d-2}}x). 
\end{equation}
Moreover, Lemma \ref{13/03/18/15:50} implies that $\widetilde{L}_{\omega,-}$ is non-negative,   
\begin{align}
\label{17/08/14/12:09}
&{\rm Ker}~\widetilde{L}_{\omega,-}={\rm span}~\{\widetilde{\Phi}_{\omega}\},
\\[6pt]
\label{18/09/06/22:02}
&\langle \widetilde{L}_{\omega,-} v , v \rangle 
\ge 
c(\omega) \|v\|_{H^{1}}^{2} \quad \mbox{for all $v\in H^{1}(\mathbb{R}^{d})$ with $\langle v, \widetilde{\Phi}_{\omega} \rangle=0$}
. 
\end{align}

Now, we are in a position to prove Theorem \ref{17/08/03/11:47}. 

\begin{proof}[Proof of Theorem \ref{17/08/03/11:47}]
It suffices to prove that there exists a nontrivial real-valued function $f \in H^{4}(\mathbb{R}^{d})$ and $\mu>0$ such that 
\begin{equation}\label{17/08/05/11:50}
\widetilde{L}_{\omega,-}^{\frac{1}{2}}\widetilde{L}_{\omega,+}\widetilde{L}_{\omega,-}^{\frac{1}{2}}
f= -\mu^{2} f
.
\end{equation} 
Indeed, putting   
\begin{equation}\label{17/08/05/17:02}
g 
:=\widetilde{L}_{\omega,-}^{\frac{1}{2}}
f 
-
i\frac{1}{\mu} \widetilde{L}_{\omega,+} \widetilde{L}_{\omega,-}^{\frac{1}{2}}
f,
\end{equation}
and using \eqref{17/08/05/11:50}, we obtain 
\begin{equation}\label{17/08/05/11:53} 
-i\widetilde{\mathcal{L}}_{\omega} g 
=
-i \widetilde{L}_{\omega,+} \widetilde{L}_{\omega,-}^{\frac{1}{2}}f
-
\frac{1}{\mu}\widetilde{L}_{\omega,-} \widetilde{L}_{\omega,+}
\widetilde{L}_{\omega,-}^{\frac{1}{2}}f
=
\mu 
\Big(
-i\frac{1}{\mu} \widetilde{L}_{\omega,+} \widetilde{L}_{\omega,-}^{\frac{1}{2}}f +
\widetilde{L}_{\omega,-}^{\frac{1}{2}}f 
\Big)
=
\mu g 
,
\end{equation}
which together with \eqref{18/09/06/20:42} yields the desired result. Note here that \eqref{17/08/05/11:50} tells us that $\widetilde{L}_{\omega,-}^{\frac{1}{2}}f$ is nontrivial and therefore so is $g$. 

We shall prove \eqref{17/08/05/11:50}. To this end, we introduce 
\begin{equation}\label{18/09/06/04:47}
\nu_{\omega}:=
\inf\Bigm\{
\frac{\langle 
\widetilde{L}_{\omega,-}^{\frac{1}{2}} \widetilde{L}_{\omega,+} \widetilde{L}_{\omega,-}^{\frac{1}{2}}f, f
\rangle}{\| f \|_{L^{2}}^{2}}
\colon f\in H^{4}(\mathbb{R}^{d}),~\langle f, \widetilde{\Phi}_{\omega} \rangle=0  
\Bigm\}
. 
\end{equation}
We can verify that the minimizer for the problem \eqref{18/09/06/04:47} becomes an eigenfunction of the operator $\widetilde{L}_{\omega,-}^{\frac{1}{2}}\widetilde{L}_{\omega,+}\widetilde{L}_{\omega,-}^{\frac{1}{2}}$ associated with $\nu_{\omega}$ (see, e.g., the proof of Proposition 4.1 of \cite{AIKN3}). Hence,  what we need to prove is that: $\nu_{\omega} \in (-\infty, 0)$; and  
 the problem \eqref{18/09/06/04:47} has a minimizer. 

First, we shall show 
\begin{equation}\label{17/08/05/17:26}
\nu_{\omega} \in (-\infty, 0).
\end{equation}
Let $\{f_{n}\}$ be a minimizing sequence for $\nu_{\omega}$. Note that $\nu_{\omega}$ is bounded from above; for instance, 
\begin{equation}\label{18/09/12/06:58}
\nu_{\omega}
\le 
\frac{\langle 
\widetilde{L}_{\omega,-}^{\frac{1}{2}}\widetilde{L}_{\omega,+}\widetilde{L}_{\omega,-}^{\frac{1}{2}}f_{1}, f_{1}
\rangle}{\| f_{1}\|_{L^{2}}^{2} } \le C(\omega). 
\end{equation}
Moreover, we can verify  that  for any $n \ge 1$, 
\begin{equation}\label{15/04/23/22:20}
\begin{split}
&
\frac{\langle 
\widetilde{L}_{\omega,-}^{\frac{1}{2}}\widetilde{L}_{\omega,+}\widetilde{L}_{\omega,-}^{\frac{1}{2}}f_{n}, f_{n}
\rangle}{\| f_{n}\|_{L^{2}}^{2} }
=
\frac{\langle \widetilde{L}_{\omega, +} \widetilde{L}^{1/2}_{\omega, -}f_{n}, \widetilde{L}^{1/2}_{\omega, -}f_{n} \rangle}{\| f_{n}\|_{L^{2}}^{2} }
\\[6pt]
&=
\frac{\langle \widetilde{L}_{\omega, -} \widetilde{L}^{1/2}_{\omega, -}f_{n}, \widetilde{L}^{1/2}_{\omega, -}f_{n} \rangle}{\| f_{n}\|_{L^{2}}^{2} }
-
(p-1)
\frac{\langle \widetilde{\Phi}_{\omega}^{p-1} \widetilde{L}^{1/2}_{\omega, -}f_{n}, \widetilde{L}^{1/2}_{\omega, -}f_{n} \rangle}{\| f_{n}\|_{L^{2}}^{2} }
-
\frac{4}{d-2}
\frac{\langle \widetilde{\Phi}_{\omega}^{\frac{4}{d-2}} \widetilde{L}^{1/2}_{\omega, -}f_{n}, \widetilde{L}^{1/2}_{\omega, -}f_{n} \rangle}{\| f_{n}\|_{L^{2}}^{2} }
\\[6pt]
&\ge 
\frac{ \| \widetilde{L}_{\omega, -} f_{n} \|_{L^{2}}^{2}}{\| f_{n}\|_{L^{2}}^{2} }
-
(p-1)\| \widetilde{\Phi}_{\omega}\|_{L^{\infty}}^{p-1}
\frac{\langle  \widetilde{L}_{\omega,-} f_{n}, f_{n} \rangle}{\| f_{n}\|_{L^{2}}^{2} }
-
\frac{4}{d-2} \|\widetilde{\Phi}_{\omega}\|_{L^{\infty}}^{\frac{4}{d-2}}
\frac{\langle  \widetilde{L}_{\omega, -}f_{n}, f_{n} \rangle}{\| f_{n}\|_{L^{2}}^{2} }
\\[6pt]
&\ge 
\frac{1}{2}
\frac{ \| \widetilde{L}_{\omega, -} f_{n} \|_{L^{2}}^{2}}{\| f_{n}\|_{L^{2}}^{2} }
-
C(\omega)
. 
\end{split}
\end{equation}
Hence, \eqref{15/04/23/22:20} together with \eqref{18/09/12/06:58} shows that for any $n\ge 1$, 
\begin{equation}\label{18/09/11/14:35}
\frac{ \| \widetilde{L}_{\omega, -} f_{n} \|_{L^{2}}}{\| f_{n}\|_{L^{2}}}
\le 
C(\omega).
\end{equation}
A computation similar to \eqref{15/04/23/22:20} also shows that 
\begin{equation}\label{18/09/11/14:36}
\begin{split}
&
\frac{\langle 
\widetilde{L}_{\omega,-}^{\frac{1}{2}}\widetilde{L}_{\omega,+}\widetilde{L}_{\omega,-}^{\frac{1}{2}}f_{n}, f_{n}
\rangle}{\| f_{n}\|_{L^{2}}^{2} }
\\[6pt]
&\ge 
\frac{ \| \widetilde{L}_{\omega, -} f_{n} \|_{L^{2}}^{2}}{\| f_{n}\|_{L^{2}}^{2} }
-
(p-1)\| \widetilde{\Phi}_{\omega}\|_{L^{\infty}}^{p-1}
\frac{\langle  \widetilde{L}_{\omega,-} f_{n}, f_{n} \rangle}{\| f_{n}\|_{L^{2}}^{2} }
-
\frac{4}{d-2} \|\widetilde{\Phi}_{\omega}\|_{L^{\infty}}^{\frac{4}{d-2}}
\frac{\langle  \widetilde{L}_{\omega, -}f_{n}, f_{n} \rangle}{\| f_{n}\|_{L^{2}}^{2} }
\\[6pt]
&\ge -C(\omega)
\frac{ \| \widetilde{L}_{\omega, -} f_{n} \|_{L^{2}}}{\| f_{n}\|_{L^{2}}}
. 
\end{split}
\end{equation}
Thus, this together with \eqref{15/04/23/22:20} shows $\nu_{\omega}>-\infty$. In order to prove $\nu_{\omega}<0$, we use a function $Z \in H^{2}(\mathbb{R}^{d})$ with the following property: 
\begin{align}
\label{18/09/12/21:39}
\langle Z, W \rangle&=0,
\\[6pt]
\label{17/08/06/15:22}
-E :=\langle L_{+}Z, Z \rangle &<0.
\end{align}
The existence of such a function is proved by Duyckaerts and Merle (see (7.16) of \cite{DM}). Furthermore, fix a smooth even function $\chi$ on $\mathbb{R}$ such that $\chi(r)=1$ for $0\le r \le 1$ and $\chi(r)=0$ for $r\ge 2$, and define  \begin{equation}\label{18/09/12/21:47}
Z_{R}(x):=\chi\Big(\frac{|x|}{R}\Big)Z(x). 
\end{equation}
Let $\varepsilon>0$ be a small constant to be specified later, and let $R(\varepsilon)>0$ be a constant such that 
\begin{align}
\label{18/09/12/22:09}
\big| \langle Z_{R(\varepsilon)}, W \rangle \big| &\le \varepsilon, 
\\[6pt]
\label{18/09/12/21:49}
\langle L_{+}Z_{R(\varepsilon)}, Z_{R(\varepsilon)}\rangle &< -\frac{E}{2}. 
\end{align}
Then, we define    
\begin{equation}\label{17/08/06/16:00}
g_{\omega,\varepsilon}:= \kappa_{\omega,\varepsilon} \widetilde{\Phi}_{\omega}+Z_{R(\varepsilon)} ,
\end{equation}
where $\kappa_{\omega, \varepsilon}$ is chosen so that $\langle g_{\omega,\varepsilon},\widetilde{\Phi}_{\omega}\rangle=0$, namely 
\begin{equation}\label{17/08/06/16:50}
\kappa_{\omega,\varepsilon}
:= 
-\frac{\langle \widetilde{\Phi}_{\omega}, Z_{R(\varepsilon)} \rangle}{\|\widetilde{\Phi}_{\omega} \|_{L^{2}}^{2}}
.
\end{equation}
Notice from \eqref{17/08/14/12:09} that $(\widetilde{L}_{\omega,-})^{-1}g_{\omega,\varepsilon}$ is well-defined. Thus, it suffices for $\nu_{\omega}<0$ to show \begin{equation}\label{17/08/13/17:31}
\lim_{\omega \to \infty}
\langle \widetilde{L}_{\omega,+}g_{\omega,\varepsilon}, g_{\omega,\varepsilon} \rangle
\le -\frac{E}{100}.
\end{equation}
Let us prove \eqref{17/08/13/17:31}. Observe that 
\begin{equation}\label{17/08/13/14:26}
\begin{split}
&\langle \widetilde{L}_{\omega,+}g_{\omega, \varepsilon}, g_{\omega, \varepsilon} \rangle
\\[6pt]
&=
\kappa_{\omega, \varepsilon}^{2} 
\langle \widetilde{L}_{\omega,+}\widetilde{\Phi}_{\omega}, \widetilde{\Phi}_{\omega} \rangle+
2 \kappa_{\omega, \varepsilon}
\langle \widetilde{L}_{\omega,+}\widetilde{\Phi}_{\omega, \varepsilon}, Z_{R(\varepsilon)} \rangle
+
\langle \widetilde{L}_{\omega,+}Z_{R(\varepsilon)}, Z_{R(\varepsilon)}  \rangle
.
\end{split} 
\end{equation}
We consider the first term on the right-hand side of \eqref{17/08/13/14:26}. Note that 
\begin{equation}\label{17/08/13/15:40}
\widetilde{L}_{\omega,+}\widetilde{\Phi}_{\omega}
=
-
(p-1)\beta_{\omega} \widetilde{\Phi}_{\omega}^{p}
-
\frac{4}{d-2} \widetilde{\Phi}_{\omega}^{\frac{d+2}{d-2}}
.
\end{equation}
Hence, we see that  
\begin{equation}\label{17/08/13/15:43}
\kappa_{\omega, \varepsilon}^{2} 
\langle \widetilde{L}_{\omega,+}\widetilde{\Phi}_{\omega}, \widetilde{\Phi}_{\omega} 
\rangle
=
-\kappa_{\omega, \varepsilon}^{2} (p-1) \beta_{\omega} \|\widetilde{\Phi}_{\omega}\|_{L^{p+1}}^{p+1}
-\kappa_{\omega, \varepsilon}^{2} 
\frac{4}{d-2} \|\widetilde{\Phi}_{\omega}\|_{L^{2^{*}}}^{2^{*}}
< 0. 
\end{equation} 
We move on to the second and the third terms on the right-hand side of \eqref{17/08/13/14:26}. Notice from the \eqref{18/09/12/22:09}, Lemma \ref{18/09/05/01:01} and the compact embedding $\dot{H}^{1}(B_{R(\varepsilon)})\hookrightarrow  L^{2}(B_{R(\varepsilon)}) $ that for any sufficiently large $\omega$ depending on $\varepsilon$, 
\begin{equation}\label{18/09/12/22:13}
\big| \langle \widetilde{\Phi}_{\omega}, Z_{R(\varepsilon)} \rangle \big|
\le 
\big| \langle W, Z_{R(\varepsilon)} \rangle \big|
+
\big| \langle \widetilde{\Phi}_{\omega}-W, Z_{R(\varepsilon)} \rangle \big|
\le 2 \varepsilon.
\end{equation}
Furthermore, it follows from \eqref{18/09/12/22:13}, Lemma \ref{18/09/05/01:01} and the compact embedding $\dot{H}^{1}(B_{1})\hookrightarrow  L^{2}(B_{1})$ 
 that for any sufficiently large $\omega$ depending on $\varepsilon$, 
\begin{equation}\label{17/08/13/17:13}
|\kappa_{\omega, \varepsilon} |
\le 
\frac{2 \varepsilon }{ \|\widetilde{\Phi}_{\omega}\|_{L^{2}(B_{1})}^{2}} 
\le 
\frac{2\varepsilon}{\| W \|_{L^{2}(B_{1})}^{2}}
. 
\end{equation}
Then, we see from Lemmas \ref{proposition:2.3} 
\ref{18/09/05/01:05},  \eqref{17/08/13/17:13} and H\"{o}lder inequality that 
\begin{equation}\label{17/08/13/16:59}
\begin{split}
\lim_{\omega \to \infty}2 \kappa_{\omega}
\big| 
\langle \widetilde{L}_{\omega,+}\widetilde{\Phi}_{\omega}, Z_{R(\varepsilon)}  \rangle
\big| 
&\le 
2(p-1) 
\lim_{\omega \to \infty} \big| \kappa_{\omega,\varepsilon}  \beta_{\omega}  
\langle \widetilde{\Phi}_{\omega}^{p}, Z_{R(\varepsilon)}  \rangle
\big| 
\\[6pt]
&\quad +\frac{8}{d-2} \lim_{\omega \to \infty}  
\big| \kappa_{\omega,\varepsilon} 
\langle \widetilde{\Phi}_{\omega}^{\frac{d+2}{d-2}}, Z_{R(\varepsilon)}  \rangle\big|
\\[6pt]
&\lesssim \varepsilon  
\|W\|_{L^{2^{*}}}^{\frac{d+2}{d-2}} \| Z \|_{L^{2^{*}}}
\lesssim \varepsilon
.
\end{split} 
\end{equation}
We also see from \eqref{18/09/12/21:49} and Lemma \ref{proposition:2.3} that 
\begin{equation}\label{17/08/13/14:37}
\begin{split}
&\lim_{\omega \to \infty}
\langle \widetilde{L}_{\omega,+}Z_{R(\varepsilon)}, Z_{R(\varepsilon)} \rangle
\\[6pt]
&=
\langle L_{+}Z_{R(\varepsilon)}, Z_{R(\varepsilon)} \rangle
+
\lim_{\omega \to \infty} \alpha_{\omega}\| Z_{R(\varepsilon)}  \|_{L^{2}}^{2}
\\[6pt]
&\quad -
p \lim_{\omega \to \infty} \beta_{\omega} 
\langle \widetilde{\Phi}_{\omega}^{p-1}Z_{R(\varepsilon)} , Z_{R(\varepsilon)}  \rangle   
-
\frac{d+2}{d-2} \lim_{\omega \to \infty} 
\langle 
(\widetilde{\Phi}_{\omega}^{\frac{4}{d-2}}-W^{\frac{4}{d-2}}) 
Z_{R(\varepsilon)} , Z_{R(\varepsilon)}  \rangle   
\\[6pt]
&\le 
-\frac{E}{2}. 
\end{split} 
\end{equation}
Putting \eqref{17/08/13/14:26}, \eqref{17/08/13/15:43}, \eqref{17/08/13/16:59} and \eqref{17/08/13/14:37} together, we obtain 
\begin{equation}\label{18/09/12/22:47}
\lim_{\omega \to \infty}\langle \widetilde{L}_{\omega,+}g_{\omega, \varepsilon}, g_{\omega, \varepsilon} \rangle
\le 
C \varepsilon - \frac{E}{2}
.
\end{equation}
Thus, taking $\varepsilon \ll E$, we obtain the desired estimate \eqref{17/08/13/17:31}: hence \eqref{17/08/05/17:26} is true.

Finally, we shall prove the existence of minimizer for the problem \eqref{18/09/06/04:47}. Let $\{f_{n}\}$ be a minimizing sequence for $\nu_{\omega}$: hence, 
\begin{align}
\label{15/04/24/8:43}
\langle f_{n}, \widetilde{\Phi}_{\omega} \rangle &= 0
\quad \mbox{for all $n\ge 1$}, 
\\[6pt] 
\label{18/09/08/08:42}
\lim_{n \to \infty} 
\frac{\langle \widetilde{L}_{\omega, -}^{\frac{1}{2}} \widetilde{L}_{\omega, +} \widetilde{L}_{\omega, -}^{\frac{1}{2}}f_{n}, f_{n} \rangle}{
\| f_{n}\|_{L^{2}}^{2}} 
&= 
\nu_{\omega}
. 
\end{align}
Furthermore, we put 
\begin{equation}\label{18/09/08/09:43}
g_{n}:=\frac{f_{n}}{\| f_{n}\|_{L^{2}}}.
\end{equation}
Note that 
\begin{align}
\label{18/09/08/12:29}
\langle g_{n}, \widetilde{\Phi}_{\omega} \rangle
&\equiv 0 
, 
\\[6pt]
\label{18/09/08/13:18}
\| g_{n}\|_{L^{2}}
&\equiv 1
,  
\\[6pt]
\label{18/09/08/10:57}
\lim_{n \to \infty} 
\langle \widetilde{L}_{\omega, -}^{\frac{1}{2}} \widetilde{L}_{\omega, +} \widetilde{L}_{\omega, -}^{\frac{1}{2}}g_{n}, g_{n} \rangle
&=
\nu_{\omega}
.
\end{align}
Since $\| \widetilde{L}_{\omega,-}^{\frac{1}{2}} \widetilde{\Phi}_{\omega}\|_{L^{2}}^{2}
=
\langle \widetilde{L}_{\omega,-} \widetilde{\Phi}_{\omega}, \widetilde{\Phi}_{\omega} \rangle=0 $, we also have 
\begin{equation}\label{18/09/08/13:47}
\langle \widetilde{L}_{\omega,-}^{\frac{1}{2}}g_{n},  \widetilde{\Phi}_{\omega}
 \rangle=0.
\end{equation}
We see from \eqref{18/09/06/22:02}, \eqref{15/04/24/8:43}, the Cauchy-Schwartz estimate and \eqref{18/09/11/14:35} that for any $n\ge 1$,
\begin{equation}\label{18/09/08/13:52}
\|g_{n}\|_{H^{1}}^{2} 
=
\frac{\|f_{n} \|_{H^{1}}^{2}}{\| f_{n}\|_{L^{2}}^{2}} 
\le 
C(\omega).
\end{equation}
Then, we see from \eqref{18/09/06/22:02}, \eqref{18/09/08/13:47}, \eqref{18/09/08/10:57}, \eqref{18/09/08/13:18}, \eqref{eq:2.13} and $\nu_{\omega}<0$ that for any sufficiently large $n\ge 1$,
\begin{equation}\label{15/04/24/8:52}
\begin{split} 
&\|\widetilde{L}_{\omega,-}^{\frac{1}{2}}g_{n}\|_{H^{1}}^{2} 
\le 
C(\omega)
\langle 
\widetilde{L}_{\omega,-} \widetilde{L}_{\omega,-}^{\frac{1}{2}}g_{n}, \widetilde{L}_{\omega,-}^{\frac{1}{2}} g_{n} \rangle 
\\[6pt]
&=
C(\omega)
\langle \widetilde{L}_{\omega,+} \widetilde{L}_{\omega,-}^{\frac{1}{2}}g_{n}, \widetilde{L}_{\omega,-}^{\frac{1}{2}} g_{n} \rangle 
+
C(\omega)
(p-1)\beta_{n} 
\langle \widetilde{\Phi}_{\omega}^{p-1} \widetilde{L}_{\omega,-}^{\frac{1}{2}}g_{n}, \widetilde{L}_{\omega,-}^{\frac{1}{2}} g_{n} \rangle 
\\[6pt]
&\quad +
C(\omega)\frac{4}{d-2} 
\langle \widetilde{\Phi}_{\omega}^{\frac{4}{d-2}} \widetilde{L}_{\omega,-}^{\frac{1}{2}}g_{n}, \widetilde{L}_{\omega,-}^{\frac{1}{2}} g_{n} \rangle 
\\[6pt]
&\le \frac{1}{2}C(\omega)\nu_{\omega}
+ 
C(\omega) \beta_{\omega} \|\widetilde{\Phi}_{\omega}\|_{L^{\infty}}^{p-1}
\|\widetilde{L}_{\omega,-}^{\frac{1}{2}} g_{n} \|_{L^{2}}^{2} 
+
C(\omega) \|\widetilde{\Phi}_{\omega}\|_{L^{\infty}}^{\frac{4}{d-2}}
\|\widetilde{L}_{\omega,-}^{\frac{1}{2}} g_{n} \|_{L^{2}}^{2} 
\le C(\omega)
.
\end{split}
\end{equation}

The standard compactness theory together with \eqref{18/09/08/13:52} and \eqref{15/04/24/8:52} shows that there exist some subsequence of $\{g_{n}\}$ (still denoted by the same symbol) and functions $g_{\infty} \in H^{1}(\mathbb{R}^{d})$ and $h_{\infty} \in H^{1}(\mathbb{R}^{d})$ such that 
\begin{align}
\label{18/09/08/11:09}
\lim_{n \to \infty} g_{n}
&= g_{\infty}
\quad \mbox{weakly in $H^{1}(\mathbb{R}^{d})$},
\\[6pt]
\label{18/09/08/15:30}
\lim_{n \to \infty} (\widetilde{L}_{\omega,-})^{\frac{1}{2}}g_{n} 
&= 
h_{\infty}
\quad \mbox{weakly in $H^{1}(\mathbb{R}^{d})$}
. 
\end{align}
The weak convergence \eqref{18/09/08/15:30} together with \eqref{18/09/08/13:47} shows 
\begin{equation}\label{18/09/08/15:39}
\langle h_{\infty}, \widetilde{\Phi}_{\omega} \rangle =0. 
\end{equation}
Hence, $(\widetilde{L}_{\omega,-})^{-\frac{1}{2}}h_{\infty}=(\widetilde{L}_{\omega,-})^{\frac{1}{2}} (\widetilde{L}_{\omega,-})^{-1}h_{\infty}$ is well-defined. Furthermore, the uniqueness of weak limit implies  
\begin{equation}\label{18/09/08/15:55}
g_{\infty}=(\widetilde{L}_{\omega,-})^{-\frac{1}{2}}h_{\infty}
.
\end{equation}
We see from \eqref{18/09/08/15:55} and the lower semicontinuity of the weak limit that 
\begin{equation}\label{18/09/08/12:41}
\begin{split}
&
\langle  \widetilde{L}_{\omega, +} \widetilde{L}_{\omega, -}^{\frac{1}{2}}g_{\infty}, \widetilde{L}_{\omega, -}^{\frac{1}{2}} g_{\infty} \rangle 
=
\langle  \widetilde{L}_{\omega, +} h_{\infty}, h_{\infty} \rangle
\\[6pt]
&=
\|\nabla  h_{\infty} \|_{L^{2}}^{2}
+
\alpha_{\omega} \|h_{\infty} \|_{L^{2}}^{2}
-
p \beta_{\omega} \int_{\mathbb{R}^{d}} \widetilde{\Phi}_{\omega}^{p-1} h_{\infty}^{2}
-
\frac{d+2}{d-2}
\int_{\mathbb{R}^{d}} \widetilde{\Phi}_{\omega}^{p-1} h_{\infty}^{2}
\\[6pt]
&\le 
\liminf_{n\to \infty} 
\big\{\|\nabla  \widetilde{L}_{\omega, -}^{\frac{1}{2}}g_{n} \|_{L^{2}}^{2}
+
\alpha_{\omega} \|\widetilde{L}_{\omega, -}^{\frac{1}{2}}g_{n} \|_{L^{2}}^{2}
\big\}
\\[6pt]
&\quad -
\lim_{n\to \infty} 
\Big\{ p \beta_{\omega} \int_{\mathbb{R}^{d}} \widetilde{\Phi}_{\omega}^{p-1} 
 ( \widetilde{L}_{\omega, -}^{\frac{1}{2}}g_{n})^{2}
+
\frac{d+2}{d-2}
\int_{\mathbb{R}^{d}} \widetilde{\Phi}_{\omega}^{p-1}
( \widetilde{L}_{\omega, -}^{\frac{1}{2}}g_{n})^{2}
\Big\}
\\[6pt]
&= 
\liminf_{n\to \infty}\langle \widetilde{L}_{\omega, +} \widetilde{L}_{\omega, -}^{\frac{1}{2}}g_{n}, \widetilde{L}_{\omega, -}^{\frac{1}{2}} g_{n} \rangle 
\end{split} 
\end{equation}
Furthermore, this together with \eqref{18/09/08/13:18} and \eqref{18/09/08/10:57} shows 
\begin{equation}\label{15/04/24/9:13}
\langle  \widetilde{L}_{\omega, +} \widetilde{L}_{\omega, -}^{\frac{1}{2}}g_{\infty}, \widetilde{L}_{\omega, -}^{\frac{1}{2}} g_{\infty} \rangle 
\le 
\nu_{\omega}.
\end{equation}
Thus, we find that the existence of minimizer follows from  
\begin{equation} \label{15/04/24/9:18}
\| g_{\infty} \|_{L^{2}}^{2}=1
. 
\end{equation}
Let us prove this. We see from \eqref{18/09/08/15:55},  the lower semicontinuity of weak limit, \eqref{18/09/08/15:30} and \eqref{18/09/08/13:18} that 
\begin{equation}\label{15/04/24/9:17}
\|  g_{\infty} \|_{L^{2}}
\le 
\liminf_{n\to \infty} \|  g_{n} \|_{L^{2}}
=1
. 
\end{equation}
Suppose the contrary that 
\begin{equation}\label{18/09/08/23:10}
\| g_{\infty} \|_{L^{2}} <1 . 
\end{equation}
Note here that if follows from \eqref{15/04/24/9:13} and $\nu_{\omega}<0$ that 
 $g_{\infty}$ must be non-trivial. Put  $\lambda_{\infty} := \| g_{\infty} \|_{L^{2}}^{-1}$, so that, by the hypothesis \eqref{18/09/08/23:10}, $\lambda_{\infty}>1$ and 
\begin{equation}\label{18/09/13/21:10}
\| \lambda_{\infty} g_{\infty} \|_{L^{2}}
=1
. 
\end{equation}
Note here that the weak convergence \eqref{18/09/08/11:09} together with \eqref{18/09/08/12:29} implies  
\begin{equation}\label{18/09/08/23:19}
\langle g_{\infty}, \widetilde{\Phi}_{\omega} \rangle=0. 
\end{equation} 
Hence, it follows from the definition of $\nu_{\omega}$ (see \eqref{18/09/06/04:47}) and \eqref{18/09/13/21:10} that 
\begin{equation} \label{15/04/24/9:20}
\begin{split}
\nu_{\omega} 
&\le 
\langle (\widetilde{L}_{\omega,-})^{\frac{1}{2}} \widetilde{L}_{\omega, +} (\widetilde{L}_{\omega,-})^{\frac{1}{2}}(\lambda_{\infty}g_{\infty}), 
(\lambda_{\infty}g_{\infty}) \rangle
\\[6pt]
&= 
\lambda_{\infty}^{2}
\langle (\widetilde{L}_{\omega,-})^{\frac{1}{2}} \widetilde{L}_{\omega, +} 
(\widetilde{L}_{\omega,-})^{\frac{1}{2}} g_{\infty}, 
g_{\infty} \rangle.
\end{split} 
\end{equation}
On the other hand, it follows from \eqref{15/04/24/9:13}, $\nu_{\omega}<0$ and $\lambda_{\infty}>1$ that 
\begin{equation}\label{18/09/08/23:31}
\langle  \widetilde{L}_{\omega, +} \widetilde{L}_{\omega, -}^{\frac{1}{2}}g_{\infty}, \widetilde{L}_{\omega, -}^{\frac{1}{2}} g_{\infty} \rangle 
\le 
\nu_{\omega} < \lambda_{\infty}^{-2}\nu_{\omega},
\end{equation}
which contradicts \eqref{15/04/24/9:20}. Thus, we have derived \eqref{15/04/24/9:17}. 
\end{proof}

%%%%%%%%%%%%%%%%%%%%%%%%%%%%%%%%%%%%%%%%%%%%%%%%%%%%%%%%%

\section{Proof of Theorem \ref{18/09/09/17:09}}\label{18/09/09/17:30}

In this section, we prove Theorem \ref{18/09/09/17:09}, namely the nondegeneracy of $\Phi_{\omega}$ in $H_{\rm rad}^{1}(\mathbb{R}^{d})$ for all sufficiently large $\omega$. 
\par  
As well as  \cite{CG}, we consider the equation of the form 
\begin{equation}\label{msp}
- \Delta \Psi + \alpha \Psi 
- \varepsilon |\Psi|^{p-1} \Psi 
- |\Psi|^{\frac{4}{d-2}} \Psi = 0 
\quad \mbox{in $\mathbb{R}^{d}$},   
\end{equation}
where $d \geq 3$, $\alpha>0$, $\varepsilon >0$, and $1 < p < 5$. The action associated with \eqref{msp}, say $\mathcal{S}_{\alpha,\varepsilon}$, is given by  
\begin{equation}
\mathcal{S}_{\alpha, \varepsilon}(u) := \frac{1}{2} \|\nabla u\|_{L^{2}}^{2} 
+ \frac{\alpha}{2} \|u\|_{L^{2}}^{2} - \frac{\varepsilon}{p+1} \|u\|_{L^{p+1}}^{p+1} 
- \frac{1}{2^{*}} \|u\|_{L^{2^{*}}}^{2^{*}}
. 
\end{equation}

A unique existence result of ground state to \eqref{msp} was obtained by Coles and Gustafson (see  Theorem 1.2 and Theorem 1.7 of \cite{CG}): 
\begin{theorem}\label{CG-1}
Assume that $d=3$ and $3 < p < 5$. Then, there exists $\varepsilon_{0} >0$ with the following property: for each $0 < \varepsilon < \varepsilon_{0}$, there exist a constant $\alpha(\varepsilon) >0$ such that:
\begin{equation}\label{relation-parameter}
\alpha(\varepsilon) = C_{1} \varepsilon^{2} + O(\varepsilon^{2 + \frac{1}{2}}) 
\end{equation}
where 
\begin{equation}
C_{1} := \frac{\langle \Lambda W, W^{p}\rangle^{2}}{36 \pi^{2}},
\end{equation}
and the equation \eqref{msp} with $\alpha=\alpha(\varepsilon)$ has 
 a unique positive radial solution $\Psi_{\alpha(\varepsilon)} \in H^{1}(\mathbb{R}^{d})$. Furthermore, the solution $\Psi_{\alpha(\varepsilon)}$ satisfies 
\begin{align}
\label{18/09/10/08:07}
\|\Psi_{\alpha(\varepsilon)}-W\|_{\dot{H}^{1}} & \lesssim \varepsilon^{\frac{1}{2}}, 
\\[6pt]
\label{relation-parameter2}
\|\Psi_{\alpha(\varepsilon)}-W\|_{L^{r}} &\lesssim \varepsilon^{1- \frac{3}{r}} 
\quad \mbox{for all $3 < r \leq \infty$},  
\end{align}
where the implicit constant in \eqref{relation-parameter2} depends on $r$ as well as $d$ and $p$. 
\end{theorem}
\begin{remark}\label{18/09/11/17:49}
Except for the uniqueness, the claims are true for all $2<p<5$ (see Theorem 1.2  of \cite{CG}). 
\end{remark}

Now, let $\varepsilon \in (0,\varepsilon_{0})$, and let $\alpha(\varepsilon)$ and $\Psi_{\varepsilon}$ be respectively a constant given in Theorem \ref{CG-1} and the unique solution to \eqref{msp} with $\alpha=\alpha(\varepsilon)$. Furthermore, we set 
\begin{equation} \label{scale}
\omega(\varepsilon):= \alpha(\varepsilon) \varepsilon^{-\frac{4}{2^{*}-(p+1)}}, \qquad 
M_{\varepsilon} := \varepsilon^{-\frac{1}{2^{*}-(p+1)}}, \qquad 
\Phi_{\omega(\varepsilon)}:= M_{\varepsilon} \Psi_{\alpha(\varepsilon)} 
(M_{\varepsilon}^{\frac{2}{d-2}} \cdot).    
\end{equation} 
Observe that $\Phi_{\omega(\varepsilon)}$ becomes a solution to \eqref{eq:1.1} with $\omega=\omega(\varepsilon)$ and for any $3<p<5$, 
\begin{equation}\label{18/09/11/18:12}
\lim_{\varepsilon \to 0} \omega(\varepsilon) = \infty 
.
\end{equation} 
Thus, Theorem \ref{18/09/09/17:09} follows from the following theorem:  

\begin{theorem} \label{thm-1}
Assume that $d =3$ and $3 < p < 5$. Then, there exists $\varepsilon_{1} \in (0,  \varepsilon_{0})$ ($\varepsilon_{0}$ denotes the constant given in Theorem \ref{CG-1}) such that 
$\Psi_{\alpha(\varepsilon)}$ is nondegenerate 
for all $0 < \varepsilon < \varepsilon_{1}$. 
\end{theorem}

In order to prove Theorem \ref{thm-1}, we need some preparation. Let us begin by recalling the explicit representation of the resolvent $R_{0}(\zeta):=(-\Delta -\zeta)^{-1}$ (see Section 6.23 of \cite{Lieb-Loss}): for any $\lambda >0$ and any function $u$ on $\mathbb{R}^{3}$, 
\begin{equation} \label{explicit-1}
R_{0}(-\lambda^{2}) u(x) 
= \int_{\mathbb{R}^{3}} 
\frac{e^{- \lambda |x-y|}}{4\pi |x - y|} u(y)\, dy.
\end{equation}
Using this formula \eqref{explicit-1} and the weak Young inequality (see Lieb and Loss~\cite[page 107]{Lieb-Loss} and Coles and Gustafson~\cite[Section 2.2]{CG}), we can obtain the following lemma: 
\begin{lemma}\label{resolvent-1}
Assume $d=3$. Then, for any $\lambda >0$ and any $1 \le s \le q \le \infty$ with  $3(1/s - 1/q) < 2$, 
\begin{equation}\label{resolvent-1-eq1}
\|R_{0}(-\lambda^{2})\|_{L^{s} \to L^{q}} 
\lesssim 
\lambda^{3(\frac{1}{s} - \frac{1}{q}) -2}
.
\end{equation} 
\end{lemma}

Next, let us recall the following result by Jensen and Kato (see Lemma 2.2 and 
Lemma 4.3 of \cite{Jensen-Kato}): 
\begin{lemma}\label{thm-resolvent-3}
Assume $d=3$. Let $3/2 < s < 5/2$, and let $\mathcal{B}$ denote either  $B(H^{1}_{-s}(\mathbb{R}^{3}), H^{1}_{-s}(\mathbb{R}^{3}))$ or $B(L^{2}_{-s}(\mathbb{R}^{3}), L^{2}_{-s}(\mathbb{R}^{3}))$, where $H^{1}_{-s}(\mathbb{R}^{3})$ 
and $L^{2}_{-s}(\mathbb{R}^{3})$ are the weighted Sobolev spaces endowed 
with the following norms:
\begin{equation}
\|u\|_{H^{1}_{-s}} := \|(1 + |x|^{2})^{-\frac{s}{2}}u\|_{H^{1}}, 
\quad 
\|u\|_{L^{2}_{-s}} := \|(1 + |x|^{2})^{-\frac{s}{2}}u\|_{L^{2}}. 
\end{equation}
Then, 
we have the following expansion as $\lambda \to 0$: 
\begin{equation} \label{resolvent-expansion}
(1 - 5R_{0}(-\lambda^{2})W^{4})^{-1} 
= \frac{5}{3\pi} \lambda^{-1} \langle W^{4} \Lambda W, \,
\cdot \ \rangle \Lambda W + O(1) 
\quad 
\mbox{in $\mathcal{B}$}.
\end{equation}
\end{lemma}

Using Lemma \ref{thm-resolvent-3}, Coles and Gustafson obtained the following estimate (see Lemma 2.4 of \cite{CG}): 
\begin{lemma}\label{resolvent-2}
For any $3<r \le \infty$ and any function 
$f \in L^{r}(\mathbb{R}^{3})$ satisfying $\langle W^{4} \Lambda W, f \rangle = 0$, we have 
\begin{equation} \label{eq-resolvent-1}
\|(1 - 5 R_{0}(- \lambda^{2})W^{4})^{-1}f\|_{L^{r}} \lesssim \|f\|_{L^{r}},
\end{equation}
where the implicit constant depends only on $r$. 
\end{lemma}
We are now in a position to prove Theorem \ref{thm-1}.
\begin{proof}[Proof of Theorem \ref{thm-1}]
Suppose the contrary that the claim was false. Then, we could take a sequence $\{ \varepsilon_{n} \}$ in $(0,1)$ with the following properties: 
\begin{equation}\label{18/09/11/19:51}
\lim_{n\to \infty}\varepsilon_{n}=0, 
\end{equation}
and for each $n\ge 1$, there exists a nontrivial real-valued radial function $v_{n} \in H_{\rm rad}^{1}(\mathbb{R}^{3})$ such that  
\begin{align}
\label{18/09/12/07:19}
\|\nabla v_{n}\|_{L^{2}} &= 1,
\\[6pt]
\label{eq:4.30}
-\Delta v_{n} + \alpha_n v_{n} 
&= 
\big\{  
p \varepsilon_{n} \Psi_{n}^{p-1} + 5
\Psi_{n}^{4} 		
\big\} v_{n} \quad {\rm in} \ \mathbb{R}^{3},   
\end{align}
where $\alpha_{n}$ and $\Psi_{n}$ are abbreviations to $\alpha(\varepsilon_{n})$ and $\Psi_{\alpha(\varepsilon_{n})}$, respectively. Notice from \eqref{relation-parameter} and \eqref{18/09/11/19:51} that 
\begin{equation}\label{18/09/12/07:21}
\lim_{n\to \infty}\alpha_{n}=0. 
\end{equation}
Furthermore, since $\{v_n\}$ is bounded in $\dot{H}_{\rm rad}^1(\mathbb{R}^{3})$ (see \eqref{18/09/12/07:19}), passing to a subsequence, we may assume that there exists a radial function $v_{\infty} \in \dot{H}_{\rm rad}^{1}(\mathbb{R}^{3})$  such that 
\begin{equation}\label{18/09/11/20:31} 
\lim_{n\to \infty}v_{n}=v_{\infty} \quad \text{weakly in $\dot{H}^1(\mathbb{R}^{3})$}. 
\end{equation}
We see from \eqref{relation-parameter}, \eqref{18/09/11/19:51}, \eqref{eq:4.30} and \eqref{18/09/11/20:31} that 
\begin{equation}\label{18/09/11/20:35}
L_{+} v_{\infty} = 0,
\end{equation} 
where $L_{+}$ is the linearized operator around $W$ (see \eqref{eq:1.16}). Furthermore, it follows from ${\rm Ker}\, L_{+}|_{\dot{H}_{\rm rad}^{1}} ={\rm span} \{\Lambda W \}$ that there exists $\kappa \in \mathbb{R}$ such that 
\begin{equation}\label{eq:4.33}
v_{\infty} = \kappa \Lambda W. 
\end{equation}
We shall show that $\kappa \neq 0$. Multiply \eqref{eq:4.30} by $v_{n}$ and integrate the resulting equation to obtain  
\begin{equation} \label{18/08/29/12:04}
1 = \|\nabla v_{n}\|_{L^{2}}^{2} 
\le 
\|\nabla v_{n}\|_{L^{2}}^{2} + \alpha_{n} \|v_{n}\|_{L^{2}}^{2}
\le
 p \varepsilon_{n} \int_{\mathbb{R}^{3}}  \Psi_{n}^{p-1} |v_{n}|^{2} dx + 
5 \int_{\mathbb{R}^{3}}  \Psi_{n}^{4} |v_{n}|^{2} dx. 
\end{equation}
Here, suppose the contrary that $\kappa=0$. Then, it follows from \eqref{18/09/11/20:31} and Lemma \ref{18/09/05/01:05} (together with the relationship $\Psi_{n}=\widetilde{\Phi}_{\omega(\varepsilon_{n})}$) that  
\begin{equation}
\lim_{n \to \infty} \varepsilon_{n} \int_{\mathbb{R}^{3}} \Psi_{n}^{p-1} |v_{n}|^{2} dx =
\lim_{n \to \infty} \int_{\mathbb{R}^{3}} \Psi_{n}^{4} |v_{n}|^{2} dx = 0,  
\end{equation}
which contradicts \eqref{18/08/29/12:04}. 
Thus, we have shown that $\kappa \neq 0$. 

Next, we consider $w_n(x) := x \cdot \nabla \Psi_{n} (x)$. We can verify 
 that $w_{n}$ satisfies 
\begin{equation}\label{eq:4.34}
-\Delta w_{n} + \alpha_{n} w_{n} = 
\big\{  p \varepsilon_n \Psi_{n}^{p-1} 
+ 5 \Psi_{n}^{4} \big\} w_{n} 
+ 2 \big\{ - \alpha_{n} \Psi_{n} + \varepsilon_{n} \Psi_{n}^{p} 
+ \Psi_{n}^{5} \big\}.
\end{equation}
Furthermore, multiplying \eqref{eq:4.30} by $w_{n}$ and \eqref{eq:4.34} by $v_n$, and integrating the resulting equations, we find  that 
\begin{equation}\label{eq:4.35}
\int_{\mathbb{R}^{3}}\bigm\{ - \alpha_{n} \Psi_{n} + \varepsilon_{n} \Psi_{n}^{p} + \Psi_{n}^{5} \bigm\} v_{n} \, dx = 0.
\end{equation}
Recall here that $\Psi_{n}$ satisfies 
\begin{equation}\label{17/10/04/15:56}
-\Delta \Psi_{n} + \alpha_n \Psi_{n} = 
\varepsilon_n \Psi_{n}^p + \Psi_{n}^{5} \quad {\rm in} \ \mathbb{R}^{3}. 
\end{equation}
Multiplying \eqref{17/10/04/15:56} by $v_{n}$ and \eqref{eq:4.30} by $\Psi_{n}$,  we also find that  
\begin{equation}\label{18/09/11/21:05}
\begin{split}
\int_{\mathbb{R}^{3}} 
\bigm\{ \varepsilon_n \Psi_{n}^{p} 
+ 
\Psi_{n}^{5} \bigm\} v_{n} \,d dx 
&= 
\int_{\mathbb{R}^{3}} \nabla \Psi_{n} \cdot \nabla v_{n} 
+ \alpha_{n} \Psi_{n} v_{n} \,d dx 
\\[6pt]
&= \int_{\mathbb{R}^{3}} 
\bigm\{ 
p \varepsilon_{n} \Psi_{n}^{p-1} 
+ 5 \Psi_{n}^{4} 
\bigm\} v_{n} \Psi_{n} \, dx,
\end{split}
\end{equation}
which implies 
\begin{equation}\label{18/09/11/21:07}
\langle \Psi_{n}^{5},  v_{n} \rangle 
= - \frac{(p-1)}{4} \int_{\mathbb{R}^{3}} 
\varepsilon_{n} \Psi_{n}^{p} v_{n} \, dx. 
\end{equation}
Plugging this into \eqref{eq:4.35}, we obtain 
\begin{equation} \label{proof-eq-0}
\sqrt{\alpha_{n}}
\langle  \Psi_{n} , v_{n} \rangle  
= 
\frac{5 - p}{4} \frac{\varepsilon_{n}}{\sqrt{\alpha_{n}}} 
\int_{\mathbb{R}^{3}} \Psi_{n}^{p} v_{n} \, dx.
\end{equation}
We consider the right-hand side of \eqref{proof-eq-0}. We see from \eqref{relation-parameter}, \eqref{relation-parameter2}, \eqref{eq:4.33} and an elementary computation that 
\begin{equation} \label{proof-eq-1}
\lim_{n \to \infty} \frac{\varepsilon_{n}}{\sqrt{\alpha_{n}}} 
\langle \Psi_{n}^{p},  v_{n} \rangle   
= \frac{\kappa}{\sqrt{C_{1}}} \int_{\mathbb{R}^{3}} W^{p} \Lambda W dx 
= \frac{\kappa}{\sqrt{C_{1}}} 
\left(\frac{1}{2} - \frac{3}{p+1}\right)\|W\|_{L^{p+1}}^{p+1}. 
\end{equation}
Next, we consider the left-hand side of \eqref{proof-eq-0}.  We shall show that \begin{equation}\label{proof-eq-2}
\lim_{n \to \infty} \sqrt{\alpha_{n}}
\langle \Psi_{n}, v_{n} \rangle = 0. 
\end{equation}
Note that \eqref{proof-eq-0} together with \eqref{proof-eq-2} and \eqref{proof-eq-1} yields a contradiction. Thus, all we have to do is to prove \eqref{proof-eq-2}. 
\par 
Let us prove \eqref{proof-eq-2}. Rewrite the equation \eqref{eq:4.30} as follows: 
\begin{equation}\label{18/09/11/21:28}
(- \Delta + \alpha_{n} - 5 W^{4}) v_{n} 
= p \varepsilon_{n} \Psi_{n}^{p-1} v_{n} 
+ 5 (\Psi_{n}^{4} - W^{4}) v_{n}. 
\end{equation}
We further rewrite the equation \eqref{18/09/11/21:28} in the form 
\begin{equation} \label{proof-eq-9}
v_{n} = 
(1 - 5 R(-\alpha_{n}) W^{4})^{-1} R(-\alpha_{n})
\bigm[ p \varepsilon_{n} \Psi_{n}^{p-1} v_{n} 
+ 5 (\Psi_{n}^{4} - W^{4}) v_{n}\bigm]. 
\end{equation}
Next, we decompose $\Psi_{n}$ as follows: 
\begin{equation}\label{18/09/12/07:25}
\Psi_{n} = k_{n} W^{4} \Lambda W + g_{n}, 
\end{equation}
where 
\begin{equation}\label{18/09/12/07:26}
k_{n} := \frac{\langle \Psi_{n}, W^{4} \Lambda W \rangle}
{\|W^{4} \Lambda W\|_{L^{2}}^{2}}. 
\end{equation}
Note that 
\begin{equation}\label{18/09/12/07:27}
\langle g_{n}, W^{4}\Lambda W \rangle=0. 
\end{equation}
Moreover, it follows from \eqref{relation-parameter2} in Theorem \ref{CG-1} that\begin{equation} \label{proof-eq-5}
|k_{n}| \leq \frac{\|\Psi_{n}\|_{L^{\infty}} \|W^{4} \Lambda W\|_{L^{1}}}
{\|W^{4} \Lambda W\|_{L^{2}}^{2}} \lesssim 1. 
\end{equation}

We consider the first term in the decomposition \eqref{18/09/12/07:25}. Using H\"{o}lder's inequality, \eqref{proof-eq-5}, \eqref{18/09/12/07:19} and \eqref{18/09/12/07:21}, we see that  
\begin{equation} \label{proof-eq-7}
\lim_{n\to \infty}
\sqrt{\alpha_{n}} 
|
\langle k_{n} W^{4} \Lambda W,  v_{n} \rangle 
|
\le 
\lim_{n\to \infty} 
\sqrt{\alpha_{n}} |k_{n}| \|W^{4} \Lambda W\|_{L^{\frac{6}{5}}} 
\|v_{n}\|_{L^{6}}=0
. 
\end{equation}
Next, we consider the second term in \eqref{18/09/12/07:25}. We see from \eqref{proof-eq-9} and Lemma \ref{resolvent-2} that  for any $r>3$, 
\begin{equation} \label{proof-eq-10}
\begin{split}
&\sqrt{\alpha_{n}} | \langle g_{n},  v_{n}\rangle  |
\\[6pt]
& = \sqrt{\alpha_{n}}
| 
\langle g_{n}, \, 
(1 - 5 R(-\alpha_{n}) W^{4})^{-1} 
R(-\alpha_{n})\big[ 
p \varepsilon_{n} \Psi_{n}^{p-1} v_{n} 
+ 5 (\Psi_{n}^{4} - W^{4}) v_{n} 
\big]
\rangle 
|
\\[6pt]
& = \sqrt{\alpha_{n}}
|
\langle (1 - 5 R(-\alpha_{n}) W^{4})^{-1} g_{n}, \, 
 R(-\alpha_{n})\big[ 
p \varepsilon_{n} \Psi_{n}^{p-1} v_{n} 
+ 5 (\Psi_{n}^{4} - W^{4}) v_{n}
\big]
\rangle 
|
\\[6pt]
& \lesssim 
\sqrt{\alpha_{n}} 
\|(1 - 5 R(-\alpha_{n}) W^{4})^{-1} g_{n}\|_{L^{r}}
\|R(-\alpha_{n})\big[ 
(\varepsilon_{n} \Psi_{n}^{p-1} v_{n} + (\Psi_{n}^{4} - W^{4}) v_{n}\big] 
\|_{L^{q}} 
\\[6pt]
& \lesssim 
\sqrt{\alpha_{n}} \|g_{n}\|_{L^{r}}
\Bigm\{ 
\varepsilon_{n}\|R(-\alpha_{n}) \big[ \Psi_{n}^{p-1} v_{n}\big] \|_{L^{q}} 
+ \|R(-\alpha_{n}) \big[ (\Psi_{n}^{4} - W^{4}) v_{n}) \big] \|_{L^{q}}
\Bigm\},  
\end{split}
\end{equation}
where $q$ is the H\"older conjugate of $r$, namely $1/q= 1-1/r$. Fix a number  
 $s>0$ such that $6/(2p-1) < s < 3/2$, which is possible since $p>3$. Furthermore, choose $r>0$ so that $s < q=\frac{r}{r-1} < 3/2$. Then, it follows from Lemma \ref{resolvent-1} 
and \eqref{relation-parameter2} that 
\begin{equation} \label{proof-eq-4}
\begin{split}
\sqrt{\alpha_{n}}\varepsilon_{n} 
\|R(-\alpha_{n}) \big[ \Psi_{n}^{p-1} v_{n}\big] \|_{L^{q}} 
&\lesssim 
\alpha_{n}^{\frac{3}{2}(\frac{1}{s} - \frac{1}{q})} 
\|\Psi_{n}^{p-1} v_{n} \|_{L^{s}} 
\\[6pt] 
&\lesssim
\alpha_{n}^{\frac{3}{2}(\frac{1}{s}- \frac{1}{q})} 
\|\Psi_{n}\|_{L^{\frac{6s(p-1)}{6-s}}}^{p-1}\|v_{n}\|_{L^{6}}.  
\end{split}
\end{equation}
Since $6s(p-1)/(6-s) > 3$, this estimate \eqref{proof-eq-4} together with \eqref{relation-parameter2} in Theorem \ref{CG-1} and \eqref{18/09/12/07:19} yields that 
\begin{equation} \label{proof-eq-17}
\sqrt{\alpha_{n}}\varepsilon_{n} 
\|R(-\alpha_{n}) \Psi_{n}^{p-1} v_{n} \|_{L^{q}} 
\lesssim \alpha_{n}^{\frac{3}{2}(\frac{1}{s}- \frac{1}{q})}.  
\end{equation}
On the other hand, it follows from Lemma \ref{resolvent-1} and \eqref{relation-parameter2} that   
\begin{equation} \label{proof-eq-6}
\begin{split}
&\sqrt{\alpha_{n}} \|R(-\alpha_{n}) \big[ (\Psi_{n}^{4} - W^{4}) v_{n}\big] \|_{L^{q}}
\\[6pt]
&\lesssim \alpha_{n}^{\frac{3}{2}(1- \frac{1}{q}) - \frac{1}{2}} 
\|(\Psi_{n}^{4} - W^{4})v_{n}\|_{L^{1}} 
\\[6pt]
& \lesssim 
\alpha_{n}^{1 - \frac{3}{2q}} 
\left(\|W^{3} (\Psi_{n}-W) v_{n}\|_{L^{1}} + 
\|(\Psi_{n}-W) ^{4} v_{n}\|_{L^{1}} \right) 
\\[6pt]
& \lesssim 
\alpha_{n}^{1 - \frac{3}{2q}} 
(\|\Psi_{n}-W\|_{L^{\infty}} 
\|W\|_{L^{\frac{18}{5}}}^{3} \|v_{n}\|_{L^{6}} + 
\|\Psi_{n}-W \|_{L^{\frac{24}{5}}}^{4} \|v_{n}\|_{L^{6}}) 
\\[6pt]
& \lesssim 
\alpha_{n}^{\frac{3}{2} -\frac{3}{2q}} 
\|v_{n}\|_{L^{6}} 
\lesssim 
\alpha_{n}^{\frac{3}{2}(1 -\frac{1}{q})}. 
\end{split}
\end{equation}
Putting \eqref{proof-eq-10}, \eqref{proof-eq-17} and \eqref{proof-eq-6} together,  we obtain  
\begin{equation} \label{proof-eq-8}
\lim_{n\to \infty} 
\sqrt{\alpha_{n}} | \langle g_{n}, v_{n} \rangle |
\lesssim 
\lim_{n \to \infty}
\Big\{ 
\alpha_{n}^{\frac{3}{2}(\frac{1}{s}- \frac{1}{q})} 
+
\alpha_{n}^{\frac{3}{2}(1 - \frac{1}{q})}
\Big\}
= 
0 
. 
\end{equation}
Furthermore, we find from \eqref{proof-eq-7} and 
\eqref{proof-eq-8} that 
\begin{equation}\lim_{n \to \infty} 
\sqrt{\alpha_{n}} \langle \Psi_{n}, v_{n} \rangle
= \lim_{n \to \infty} 
\sqrt{\alpha_{n}} \langle k_{n} W^{4} \Lambda W, v_{n} \rangle
+ \lim_{n \to \infty} \sqrt{\alpha_{n}} \langle g_{n}, v_{n} \rangle 
= 0.
\end{equation} 
Thus, we have proved \eqref{proof-eq-2} and completed the proof of Theorem \ref{18/09/09/17:09}. 
\end{proof}

%%%%%%%%%%%%%%%%%%%%%%%%%%%%%%%%%%%%%%%%%%%%%%%%%%%%%%%%%%%%%%%%%%%%%%%%%%%%%%%%
\section{Proof of Theorem \ref{17/08/15/11:05}}\label{17/08/15/11:02}

In this section, we give a proof of Theorem \ref{17/08/15/11:05}. We assume $\omega > \max\{\omega_{2}, \omega_{3}\}$. Hence, by Theorem \ref{17/08/03/11:47} and Theorem \ref{18/09/09/17:09}, $-i\mathscr{L}_{\omega}$ has a positive eigenvalue $\mu>0$ as an operator in $L_{\rm real}^{2}(\mathbb{R}^{d})$, and $\Phi_{\omega}$ is nondegenerate in $H_{\rm rad}^{1}(\mathbb{R}^{d})$.

Let $\mathscr{U}_{+}$ be an eigenfunction associated with the positive eigenvalue $\mu$, and put   
\begin{equation}\label{13/04/29/11:30}
\mathscr{U}_{-}:=\overline{\mathscr{U}_{+}}.
\end{equation} 
Then, it is easy to verify that  
\begin{equation}\label{15/02/03/13:17}
-i\mathcal{L}_{\omega}\mathscr{U}_{-}
=
-\mu \mathscr{U}_{-}. 
\end{equation} 
Hence, $\mathscr{U}_{-}$ is an eigenfunction of $-i\mathscr{L}_{\omega}$ associated with $-\mu$.
\par 
Note that Weyl's essential spectrum theorem together with \eqref{17/08/13/11:05} shows that for any $\omega>0$, 
\begin{equation}\label{17/08/15/13:57} 
\sigma_{ess}(L_{\omega,+})=[\omega, \infty).
\end{equation}
Moreover, since $\Phi_{\omega}$ is a positive solution to \eqref{eq:1.1}, we see that 
\begin{equation}\label{17/08/15/13:47}
\langle L_{\omega,+}\Phi_{\omega},\Phi_{\omega} \rangle 
=
-(p-1) \| \Phi_{\omega}\|_{L^{p+1}}^{p+1}
-\frac{4}{d-2}\| \Phi_{\omega}\|_{L^{2^{*}}}^{2^{*}}
<0
.
\end{equation}

We can verify that  for any $f \in H^{1}(\mathbb{R}^{d})$,
\begin{equation}\label{13/02/19/10:42}
\big| 
\langle 
L_{\omega,+} f, f
\rangle_{H^{-1},H^{1}}
\big|
+
\big| 
\langle   
L_{\omega,-} f, f
\rangle_{H^{-1},H^{1}}
\big| 
\lesssim  
\big( \omega +
\|\Phi_{\omega}\|_{L^{p+1}}^{p-1}
+
\|\Phi_{\omega}\|_{L^{2^{*}}}^{\frac{4}{d-2}} 
\big)
\|f\|_{H^{1}}^{2}
.
\end{equation}

We can derive the information on the negative eigenvalue of $L_{\omega,+}$ 
 in a way similar to Lemma 2.3 in \cite{Nakanishi-Schlag1}:
\begin{lemma}\label{13/01/01/16:22}
The linear operator $L_{\omega,+}$ has only one negative eigenvalue which is simple and $0$ is not an eigenvalue as an operator in $L_{\rm rad}^{2}(\mathbb{R}^{d})$.  
\end{lemma}

Under some condition of orthogonality, the operators $L_{\omega,+}$ and $L_{\omega,-}$ give us norms equivalent to the one of $H^{1}(\mathbb{R}^{d})$: 
\begin{lemma}\label{13/01/02/12:08}
There exists $\omega_{5}>0$ such that if $\omega >\omega_{5}$, $\mu$ is a positive eigenvalue of $i\mathscr{L}_{\omega}$ and $\mathscr{U}_{+}$ is an eigenfunction associated with $\mu$, then we have the following:
\\
{\rm (i)} For any real-valued radial function $g \in H^{1}(\mathbb{R}^{d})$ with $\langle g,\Im[\mathscr{U}_{+}]\rangle=0$,  
\begin{equation}
\label{13/01/02/12:11}
\langle L_{\omega,+}g,g \rangle_{H^{-1},H^{1}} \sim \|g\|_{H^{1}}^{2},
\end{equation}
where the implicit constant may depend on $\omega$.
\\[6pt]
{\rm (ii)} For any real-valued radial function $g \in H^{1}(\mathbb{R}^{d})$ with $\langle g,\partial_{\omega}\Phi_{\omega} \rangle=0$,  
\begin{equation}\label{13/01/02/12:12}
\langle L_{\omega,-}g,g \rangle_{H^{-1},H^{1}} \sim \|g\|_{H^{1}}^{2},
\end{equation}
where the implicit constant may depend on $\omega$. 
\end{lemma}
The proof of Lemma \ref{13/01/02/12:08} is similar to the one of Lemma B.5 in \cite{AIKN3} (see also Lemma 2.2 of \cite{Nakanishi-Schlag2}).

We prove Theorem \ref{17/08/15/11:05} by using the argument in the linear stability theory (cf. \cite{Grillakis}).  To this end, let $P_{\omega,-}$ denotes the orthogonal projection  onto $({\rm Ker}L_{\omega,-})^{\perp}$ and $R_{\omega}$ the operator defined by  
\begin{equation}\label{18/02/10/17:19}
R_{\omega}:=P_{\omega,-} L_{\omega,+}P_{\omega,-}.
\end{equation}
Since ${\rm Ker} L_{\omega,-}={\rm span}\{ \Phi_{\omega}\}$ (see Lemma \ref{13/03/18/15:50}), we find that for any $u\in H^{1}(\mathbb{R}^{d})$, 
\begin{equation}\label{18/02/10/17:33}
P_{\omega,-} u= u - (u,\Phi_{\omega})_{L^{2}}\frac{\Phi_{\omega}}{\|\Phi_{\omega}\|_{L^{2}}^{2}}.
\end{equation}
Furthermore, let $N(R_{\omega})$ denote the number of negative eigenvalues of $R_{\omega}$, and $I(-i\mathcal{L}_{\omega})$ the number of positive eigenvalues of $-i\mathcal{L}_{\omega}$. Since $L_{\omega,-}$ is non-negative (see Lemma \ref{13/03/18/15:50}), there is no negative eigenvalue of $L_{\omega,-}$, and therefore Corollary 1.1 of \cite{Grillakis} together with Theorem \ref{17/08/03/11:47} implies that 
\begin{equation}\label{18/02/10/17:45}
N(R_{\omega})=I(-i\mathcal{L}_{\omega}) \ge 1.
\end{equation}

Now, we are in a position to prove Theorem \ref{17/08/15/11:05}:  
\begin{proof}[Proof of Theorem \ref{17/08/15/11:05}]
We prove the claim by contradiction. Hence, suppose the contrary that 
for any $\nu>\omega_{2}$, there exists $\omega(\nu)>\nu$ such that 
\begin{equation}\label{18/02/10/16:59}
\frac{d}{d\omega}\mathcal{M}(\Phi_{\omega})\Big|_{\omega=\omega(\nu)}\ge 0. 
\end{equation}
Under this hypothesis, we show that there exists $\omega>\omega_{2}$ such that 
 for any $f\in ({\rm Ker}L_{\omega,-})^{\perp}$,
\begin{equation}\label{18/02/11/11:35}
0 \le \langle L_{\omega,+}f, f \rangle 
.
\end{equation}
Note that \eqref{18/02/11/11:35} implies $N(R_{\omega})=0$ and therefore we arrive at a contradiction (see \eqref{18/02/10/17:45}). Thus, \eqref{18/02/11/11:35} proves Theorem \ref{17/08/15/11:05}. 
\par 
We shall prove \eqref{18/02/11/11:35}. To this end, we employ an argument similar to \cite{Grillakis-Shatah-Strauss1}. 
\par 
Let $\nu>\omega_{2}$ be a frequency to be specified later. Then, it follows from the hypothesis \eqref{18/02/10/16:59} that there exists $\omega>\nu$ such that 
\begin{equation}\label{18/02/11/14:57}
0 \le 
\langle \Phi_{\omega}, \partial_{\omega}\Phi_{\omega}\rangle 
.
\end{equation}
We see from the nondegeneracy of $\Phi_{\omega}$ (see Theorem \ref{18/09/09/17:09} and Remark \ref{18/09/10/19:31}) that $L_{\omega,+}$ is one-to-one as an operator in $L_{\rm rad}^{2}(\mathbb{R}^{d})$. Moreover, it follows from Lemma \ref{13/01/01/16:22} that there exists a unique negative eigenvalue $-e_{\omega}$ of $L_{\omega,+}$ and an $L^{2}$-normalized eigenfunction $v_{\omega}$ associated with $-e_{\omega}$. Let $f\in ({\rm Ker}L_{\omega,-})^{\perp}$, and consider the following decompositions: 
\begin{align}
\label{18/02/11/15:31}
f&=\ a_{1}  v_{\omega} + w_{1} \quad \mbox{with $\langle v_{\omega}, w_{1} \rangle=0$},
\\[6pt]
\label{18/09/17/20:14}
\partial_{\omega}\Phi_{\omega}
&= a_{2} v_{\omega} + w_{2}
\quad \mbox{with $\langle v_{\omega}, w_{2} \rangle=0$}
.
\end{align}
Note that by Lemma \ref{13/03/18/15:50},  
\begin{equation}\label{18/09/17/20:31}
\langle f,\Phi_{\omega} \rangle=0.
\end{equation} 
Moreover, differential of the both sides of $L_{\omega,-}\Phi_{\omega}=0$ with respect to $\omega$ yields   
\begin{equation}\label{18/02/11/15:02} 
L_{\omega,+} (\partial_{\omega}\Phi_{\omega}) 
=
-\Phi_{\omega}.
\end{equation}

Now, we see from \eqref{18/09/17/20:31}, \eqref{18/02/11/15:02}, \eqref{18/02/11/15:31}, \eqref{18/09/17/20:14} and $L_{\omega,+}v_{\omega}=-e_{\omega}v_{\omega}$ that 
\begin{equation}\label{18/02/11/17:20}
\begin{split} 
0&=
\langle 
f, \Phi_{\omega} 
\rangle
=
-\langle 
f, L_{\omega,+}\partial_{\omega}\Phi_{\omega}
\rangle 
=
-\langle 
L_{\omega,+}f, \partial_{\omega}\Phi_{\omega}
\rangle 
\\[6pt]
&=
\langle a_{1} e_{\omega}v_{\omega}
-
L_{\omega,+} w_{1}, 
\, a_{2}v_{\omega}+ w_{2}  
\rangle 
=
a_{1}a_{2}e_{\omega} 
-
\langle L_{\omega,+} w_{1}, w_{2} \rangle 
. 
\end{split} 
\end{equation}
Thus, we have obtained 
\begin{equation}\label{18/02/11/15:11} 
\langle 
L_{\omega,+} w_{1}, w_{2} \rangle
= 
a_{1}a_{2} e_{\omega}
.
\end{equation}
On the other hand, it follows from \eqref{18/02/11/14:57}, \eqref{18/02/11/15:02}, \eqref{18/09/17/20:14} and $L_{\omega,+}v_{\omega}=-e_{\omega}v_{\omega}$ that 
\begin{equation}\label{18/09/17/20:23} 
\begin{split}
0
&\le 
\langle 
\Phi_{\omega}, \partial_{\omega} \Phi_{\omega} \rangle
=
-
\langle 
L_{\omega,+} \partial_{\omega} \Phi_{\omega}, \partial_{\omega} \Phi_{\omega} \rangle
\\[6pt]
&=
-\langle 
a_{2} (-e_{\omega}v_{\omega})+ L_{\omega,+}w_{2}, \, 
 a_{2}v_{\omega}+w_{2}  
\rangle 
=
a_{2}^{2} e_{\omega}
-
\langle 
L_{\omega,+} w_{2}, w_{2} \rangle
,
\end{split} 
\end{equation}
so that 
\begin{equation}\label{18/09/17/20:45} 
\langle 
L_{\omega,+} w_{2}, w_{2} \rangle
\le 
a_{2}^{2} e_{\omega}
.
\end{equation}
Furthermore, we see from the Cauchy-Schwartz inequality, \eqref{18/02/11/15:11} and \eqref{18/09/17/20:45} that   
\begin{equation}\label{18/02/11/22:45}
\begin{split}
\langle L_{\omega,+}f, f \rangle
&=
\langle 
-a_{1}e_{1}v_{\omega}+L_{\omega,+}w_{1},\,  a_{1}v_{\omega}+w_{1} 
\rangle
\\[6pt]
&=
-a_{1}^{2} e_{\omega}  
+
\langle L_{\omega,+} w_{1}, w_{1} \rangle  
=
-a_{1}^{2} e_{\omega}  
+
\frac{\| L_{\omega,+}^{\frac{1}{2}} w_{1} \|_{L^{2}}^{2}\| L_{\omega,+}^{\frac{1}{2}} w_{2} \|_{L^{2}}^{2} }{\| L_{\omega,+}^{\frac{1}{2}} w_{2} \|_{L^{2}}^{2}}
\\[6pt]
&\ge 
-a_{1}^{2}e_{\omega}  
+
\frac{\langle L_{\omega,+}^{\frac{1}{2}} w_{1}
, L_{\omega,+}^{\frac{1}{2}}  w_{2} \rangle^{2}}{\langle  L_{\omega,+}w_{2}, w_{2} \rangle}
\ge 
-a_{1}^{2} e_{\omega}  
+
a_{1}^{2} e_{\omega} =0
.
\end{split} 
\end{equation}
Thus, we have proved \eqref{18/02/11/11:35} and completed the proof. 
\end{proof}

\subsection*{Acknowledgement}
This work was done while H.K. was visiting at University of Victoria. 
H.K. thanks all members of the Department of Mathematics and Statistics for their warm hospitality. 
S.I. was supported by NSERC grant (371637-2014). 
H.K. was supported by JSPS KAKENHI Grant Number JP17K14223.

%%%%%%%%%%%%%%%%%%%%%%%%%%%%%%%%%%%%%%%%%%%%%%%%%%%%%%%%%%%%%%%%%%%%%%%%%%%%%%%

\bibliographystyle{plain}

%%%%%%%%%%%%%%%%%%%%%%%%%%%%%%%%%%%%%%%%%%%%%%%%%%%%%%%%%%%%%%%%%%%%%%%%%%%%%%%%

\end{document}